\documentclass[11pt,twoside]{amsart}
\usepackage{amsmath}
\usepackage{amssymb}
\usepackage{latexsym}
\usepackage{graphics}
\usepackage[all]{xy}
\usepackage{url} 

\usepackage{graphicx}
\usepackage{float}

\newtheorem{thm}{Theorem}[section]
\newtheorem{cor}[thm]{Corollary}
\newtheorem{lem}[thm]{Lemma}
\newtheorem{prop}[thm]{Proposition}

\newtheorem{rem}[thm]{\bf{Remark}}

\newtheorem{df}{Definition}[section]
\numberwithin{equation}{section}

\begin{document}

\leftline{ \scriptsize}

\vspace{1.3 cm}
\title
{ Symmetry classes of tensors associated with the Semi-Dihedral groups $SD_{8n}$}
\author{M. Hormozi}
\author{K. Rodtes}
\thanks{{\scriptsize
\hskip -0.4 true cm MSC(2000):Primary   20C30; Secondary 15A69
}}
\hskip -0.4 true cm

\maketitle



\begin{abstract} We discuss the existence of an orthogonal basis consisting of decomposable vectors for some symmetry classes of tensors associated with Semi-Dihedral groups $SD_{8n}$. The dimensions of these symmetry classes of tensors are also computed.

\end{abstract}

\vskip 0.2 true cm


\pagestyle{myheadings}
\markboth{\rightline {\scriptsize  Mahdi Hormozi and Kijti Rodtes}}
         {\leftline{\scriptsize }}
\bigskip
\bigskip


\vskip 0.4 true cm
\section{Introduction }
\label{}

Let $V$ be an $n$-dimensional complex inner product space and $G$ be a
permutation group on $m$ elements. Let $\chi$ be any irreducible character of
$G$. For any $\sigma \in G$, define the operator
$$
P_\sigma:  \bigotimes_{1}^{m} V \rightarrow   \bigotimes_{1}^{m} V
$$
by
\begin{equation}
  \label{e10}
 P_\sigma (v_1\otimes...\otimes v_m )= (v_{\sigma^{-1}(1)}\otimes...\otimes v_{\sigma^{-1}(m)} ).
\end{equation}

The symmetry classes of tensors associated with $G$ and $\chi$ is the image
of the symmetry operator
\begin{equation}
\label{e11}
T(G,\chi)= \frac{\chi(1)}{|G|}\sum_{\sigma \in G}\chi(\sigma)P_\sigma,
\end{equation}
and it is denoted by $V^{n}_{\chi}(G)$. We say that the tensor $T(G,\chi)(v_1\otimes\dots\otimes v_m )$ is a decomposable symmetrized tensor, and we denote it by $v_1\ast\dots\ast v_m$. We call $V^{n}_{\chi}(G)$ the symmetry class of tensors associated with $G$ and $\chi$, and the dimension of  $V^{n}_{\chi}(G)$ is
\begin{equation}
\label{important1}
\dim V^{n}_{\chi}(G)= \frac{\chi(1)}{|G|}\sum_{\sigma \in G}\chi(\sigma)n^{c(\sigma)},
\end{equation}
where $c(\sigma)$ is the number of cycles, including cycles of length one, in the disjoint cycle factorization of $\sigma$  \cite {Ma}. \\

The inner product on $V$ induces an inner product on $V_\chi (G)$ which satisfies
$$
\langle v_1\ast\dots\ast v_m,u_1\ast\dots\ast u_m  \rangle= \frac{\chi(1)}{|G|}\sum_{\sigma \in G}\chi(\sigma)\prod_{i=1}^{m}\langle v_{i},u_{\sigma(i)} \rangle.
$$
Let $\Gamma^{m}_n $ be the set of all sequences $\alpha = (\alpha_1,...,\alpha_m)$, with $1  \leq \alpha_i\leq n$. Define the action of $G$ on $\Gamma^{m}_n $ by
$$
\sigma.\alpha =(\alpha_{\sigma^{-1}(1)},..., \alpha_{\sigma^{-1}(m)} ).
$$

Let $O(\alpha)=\{\sigma.\alpha|\sigma\in G \}$ be the \textit{orbit} of $\alpha$. We write $\alpha \sim \beta $ if $\alpha$ and $\beta$  belong to the same orbit in $\Gamma^{m}_n$. Let $ \Delta$ be a
system of distinct representatives of the orbits. We denoted by $G_\alpha$ the \textit{stabilizer subgroup} of $\alpha$, i.e., $G_\alpha=\{\sigma\in G|\sigma.\alpha=\alpha \}$.  Define
$$
\Omega = \{\alpha \in \Gamma^{m}_n | \sum_{\sigma \in G_\alpha}\chi(\sigma) \neq 0   \},
$$
and put $\overline{\Delta} = \Delta \cap \Omega$.\\

Let $\{e_1, . . . , e_n \}$ be an orthonormal basis of $V$. Now let us denote by $e^{*}_{\alpha} $ the tensor $e_{\alpha_{1}}\ast\cdots\ast e_{\alpha_{m}}$. We have
$$
 \langle e^{*}_\alpha,e^{*}_\beta \rangle = \left\{\begin{array}{ll} 0 \hspace{3.8cm} \text{if} \hspace{0.5cm} \alpha \nsim \beta  \\
\frac{\chi(1)}{|G|}\sum_{\sigma \in G_\beta}\chi(\sigma h^{-1})\hspace{0.5cm} \text{ if} \hspace{0.5cm} \alpha =h.\beta.
\end{array}\right.
$$
In particular, for $\sigma_1, \sigma_2 \in G $ and 	$\gamma \in \overline{\Delta }$ we obtain\\
\begin{equation}
\label{important2}
\langle e^{*}_{\sigma_1.\gamma},e^{*}_{\sigma_2.\gamma}\rangle = \frac{\chi(1)}{|G|}\sum_{x \in \sigma_2 G_\gamma \sigma^{-1}_1}\chi(x).
\end{equation}
Moreover,  $e^{*}_{\alpha}\neq 0$ if and only if $\alpha \in \Omega$.\\

For $\alpha \in \overline{\Delta} $, $ V^{*}_\alpha =\langle e^{*}_{\sigma.\alpha}: \sigma \in G  \rangle $ is called the orbital subspace of $ V_\chi(G).$ It follows that
$$
V_\chi(G)= \bigoplus_{\alpha \in \overline{\Delta}} V^{*}_\alpha
$$
is an orthogonal direct sum. In [9] it is proved that
\begin{equation}
\label{important3}
\dim V^{*}_\alpha ~=~ \frac{\chi(1)}{|G_\alpha|}\sum_{\sigma \in G_\alpha}\chi(\sigma).
\end{equation}
Thus we deduce that if $\chi$ is a linear character, then $\dim V^{*}_\alpha =1$ and in this case the set
$$
\{e^{*}_{\alpha} | \alpha \in \overline{\Delta}   \}
$$
is an orthogonal basis of $V_\chi(G)$. A basis which consists of the decomposable symmetrized tensors $e^{*}_{\alpha}$ is called an orthogonal $\ast$-basis. If $\chi$ is not linear, it is possible that $V_\chi(G)$  has no orthogonal $\ast$-basis. The reader can find further information about the symmetry classes of tensors in  [1-8], [11-15] and [17]. In this paper we discuss the existence of an orthogonal basis consisting of decomposable vectors for some symmetry classes of tensors associated with semi-dihedral groups $SD_{8n}$. Also we compute the dimensions of these symmetry classes of tensors.

For the next section, we investigate the dimensions of symmetry classes of tensors associated with the  semi-dihedral groups of order $8n$, $SD_{8n}$, (Theorem \ref{thmdim1}, \ref{thmdim 2}) by using (\ref{important1}).  To do that, the character tables for $SD_{8n}$ and the number of cycles in the factorization of each element in $SD_{8n}$ are the main ingredients.  We calculate the explicit conjugacy classes for $SD_{8n}$ (Proposition \ref{conjuagacy}) at the first step and then the character tables are obtained separately for even $n$ (\textbf{Table I}) and odd $n$ (\textbf{Table II}). The embedding of $SD_{8n}$ to symmetric group $S_{4n}$ is explicit in Proposition \ref{embedding} and the number of cycles  have been calculated in Proposition \ref{c(g)}.

\section{\bf Semi-Dihedral groups $SD_{8n}$ }
The presentation for $SD_{8n}$ for $n\geq2$ is given by
$  SD_{8n}=<a,b \mid a^{4n}=b^{2}=1,bab=a^{2n-1}>.$
All $8n$ elements of $SD_{8n}$ may be given by
$$SD_{8n}=\{  1,a,a^{2},...,a^{4n-1},b,ba,ba^{2},...,ba^{4n-1}\}.$$
\begin{lem}\label{formula}  For $SD_{8n}$, we have the relations
\begin{enumerate}
  \item $ba^{k}=a^{(2n-1)k}b$,
  \item $a^{k}b=ba^{(2n-1)k}$,
    \item $a^{-k}=a^{4n-k}$, $a^{k}=a^{4n+k}$, $b=b^{-1}$,
  \item $(ba^{k})^{-1}=ba^{(2n+1)k}$.
\end{enumerate}
\end{lem}
\begin{proof} By relations $a^{4n}=b^{2}=1$ and $bab=a^{2n-1}$, the results follow immediately.
\end{proof}
By definition, for a group $G$ and $a\in G$, the conjugacy class of $a\in G$ is given by $$[a]=\{ gag^{-1}\mid g\in G\}.$$
Now, for $G=SD_{8n}$, we have  $[a^{r}]=\{ ga^{r}g^{-1}\mid g\in SD_{8n}\}$ for each $0\leq r\leq4n-1$.
If $g=a^{k}$ for $0\leq k\leq 4n-1$, then $ga^{r}g^{-1}=a^{k}a^{r}a^{-k}=a^{k+r-k}=a^{r}$ and hence $a^{k}$ does not give any new element for $[a^{r}]$.  But if $g=ba^{k}$ for $0\leq k\leq 4n-1$, then (by Lemma \ref{formula}),
\begin{equation*}
    ga^{r}g^{-1}=(ba^{k})a^{r}(ba^{k})^{-1}=ba^{r}b=a^{(2n-1)r}
\end{equation*}
and hence $a^{(2n-1)r} \in[a^{r}]$.  Therefore, for each $r=0,1,2,...,4n-1$,
\begin{equation}\label{class am}
    [a^{r}]=\{ a^{r},a^{(2n-1)r} \}.
\end{equation}
For the conjugacy class of $[ba^{r}]=\{ g(ba^{r})g^{-1}\mid g\in SD_{8n}\}$, for each  $0\leq r\leq4n-1$, we first consider elements $a^{k}$ for each $0\leq k \leq 4n-1$.  We see that
$a^{k}(ba^{r})a^{-k}=(a^{k}b)a^{r-k}=ba^{(2n-1)k}a^{r-k}=ba^{(2n-2)k+r}$.  For elements $ba^{k}$, where $0\leq k\leq4n-1$, we see that
\begin{equation*}
        \begin{array}{lclclcl}
       ba^{k}(ba^{r})(ba^{k})^{-1} & = & ba^{r-k}b & = & ba^{k+(2n-1)(r-k)}
        & = & ba^{(2n+2)k+(2n-1)r}.
       \end{array}
\end{equation*}
Therefore, for each $r=0,1,2,...,4n-1$,
\begin{equation}\label{bam}
    [ba^{r}]=\{ba^{(2n-2)k+r},ba^{(2n+2)k+(2n-1)r} \mid k=0,1,2,...,4n-1 \}.
\end{equation}

For fixed $r,k_{1} \in \{ 0,1,2,...,4n-1\} $, we see that $$(2n-2)k_{1}+r\equiv [(2n+2)k+(2n-1)r] \hbox{ mod } 4n$$ has a solution $k \in \{0,1,2,...,4n-1\}$ if and only if the linear Diophantine equation $$(n+1)k+2nt=(n-1)(k_{1}-r)$$ has a solution.  This always occurs since the greatest common divisor $d$ of $n+1$ and $2n$ is $1$ or $2$.  If $d=2$, then $n$ must be odd and hence $d \mid (n-1)(k_{1}-r)$.  Thus $$\{ba^{(2n-2)k+r}\mid k=0,1,2,...,4n-1 \} = \{ba^{(2n+2)k+(2n-1)r} \mid k=0,1,2,...,4n-1 \}.$$  Hence, by (\ref{bam}),
\begin{equation}\label{bam2}
    [ba^{r}]=\{ba^{(2n-2)k+r} \mid k=0,1,2,...,4n-1 \}.
\end{equation}
\begin{df} Define $C^{even}:=C_{1}\cup C^{even}_{2} \cup C^{even}_{3}$ and $C^{odd}:=C_{1}\cup C^{odd}_{2} \cup C^{odd}_{3}$, where $C_{1}:=\{0,2,4,...,2n \}$, $C^{even}_{2}:=\{1,3,5,...,n-1 \}$, $C^{even}_{3}:=\{2n+1,2n+3,2n+5,...,3n-1 \}$ and  $C^{odd}_{2}:=\{1,3,5,...,n \}$, $C^{odd}_{3}:=\{2n+1,2n+3,2n+5,...,3n \}$. Also, define $C^{\dag}_{even}:=C_{1}\setminus\{ 0,2n\}$, $C^{\dag}_{odd}:=C^{even}_{2}\cup C^{even}_{3}$, $C^{odd}_{2,3}:=C^{odd}_{2}\cup C^{odd}_{3}$ and $C_{*}^{even}:=C^{even}\setminus\{0,2n\}$, $C_{*}^{odd}:=C^{odd}\setminus\{0,n,2n,3n\}$.
\end{df}
By (\ref{class am}) and (\ref{bam2}), we obtain the following result.
\begin{prop}\label{conjuagacy} The conjugacy classes of $SD_{8n}$, $n\geq2$, are as follows:
\begin{itemize}
  \item If $n$ is even, there are $2n+3$ conjugacy classes. Precisely,
       \begin{itemize}
         \item $2$ classes of size one being $[1]=\{ 1\}$ and $[a^{2n}]=\{ a^{2n}\}$,
         \item $2n-1$ classes of size two being $[a^{r}]=\{ a^{r},a^{(2n-1)r}\}, $ where $r \in C_{*}^{even}$ and
         \item $2$ classes of size $2n$ being $[b]=\{ba^{2t} \mid t=0,1,2,...,2n-1\}$ and $[ba]=\{ba^{2t+1} \mid t=0,1,2,...,2n-1\}$.
       \end{itemize}
  \item If $n$ is odd, there are $2n+6$ conjugacy classes.  Precisely,
   \begin{itemize}
     \item $4$ classes of size one being $[1]=\{ 1\}$, $[a^{n}]=\{ a^{n}\}$, $[a^{2n}]=\{ a^{2n}\}$ and $[a^{3n}]=\{ a^{3n}\}$,
     \item $2n-2$ classes of size two being $[a^{r}]=\{ a^{r},a^{(2n-1)r}\}$,  where $r \in C_{*}^{odd}$ and
     \item $4$ classes of size $n$ being $[b]=\{ba^{4t} \mid t=0,1,2,...,n-1\}$, $[ba]=\{ba^{4t+1} \mid t=0,1,2,...,n-1\}$, $[ba^{2}]=\{ba^{4t+2} \mid t=0,1,2,...,n-1\}$ and $[ba^{3}]=\{ba^{4t+3} \mid t=0,1,2,...,n-1\}$.
   \end{itemize}
\end{itemize}
\end{prop}
\begin{proof} We first consider conjugacy classes of the form $[a^{r}]$.  By (\ref{class am}), it suffices to find the classes of size one to separate classes of size one or two.  Namely, we need to count $r \in \{0,1,2,...,4n-1 \}$ such that $r \equiv (2n-1)r \hbox{ mod }4n$.  This follows from the facts that
\begin{enumerate}
\item If $n$ is even, then $r \equiv (2n-1)r \hbox{ mod }4n$ if and only $r=0$ or $r=2n$,
\item If $n$ is odd, then $r \equiv (2n-1)r \hbox{ mod }4n$ if and only $r=0, r=n, r=2n$ or $r=3n$,
 \end{enumerate}
which are easy to prove.  To see that the conjugacy classes $[a^{r}]$'s are all different for $r \in C^{even}$ in case of even $n$  and for $r \in C^{odd}$ in case of odd $n$, it suffices to use the relations
\begin{enumerate}
  \item $(2n-1)r\equiv (4n-r)\hbox{ mod }4n$ if $r$ is even,
  \item $(2n-1)r\equiv (2n-r)\hbox{ mod }4n$  if $r$ is odd,
  \item $(2n-1)(2n+k)\equiv (4n-k)\hbox{ mod }4n$  if $k$ is odd,
\end{enumerate}
which are easy to prove.  These relations show that if $n$ is even and $r_{1},r_{2} \in C^{even}$, then $[a^{r_{1}}]\neq [a^{r_{2}}]$ if and only if $r_{1}\neq r_{2}$.  They also show that if $n$ is odd and $r_{1},r_{2} \in C^{odd}$, then $[a^{r_{1}}]\neq [a^{r_{2}}]$ if and only if $r_{1}\neq r_{2}$. \\

Next, we consider conjugacy classes of the form $[ba^{r}]$ separately in the even case and odd case.  For even $n$ and $\epsilon=0,1$, to show that $\{ba^{(2n-2)k+\epsilon} \mid k=0,1,2,...,4n-1 \}=\{ba^{2t+\epsilon} \mid t=0,1,2,...,2n-1\}$, it is enough to show that for each $t \in \{0,1,2,..., 2n-1\}$, there is $k \in \{0,1,2,...,4n-1 \}$ such that $[(2n-2)k+\epsilon]\equiv [2t+\epsilon] \hbox{ mod } 4n$.  This is equivalent to checking that the linear Diophantine equation $(n-1)k+2ns=t$ has a solution.  This is obvious because $\gcd(n-1,2n)=1$ (since $n$ is even) always divides $t$. \\

 For odd $n=2n_{0}-1$ and $\varepsilon=0,1,2,3$, to show that $\{ba^{(2n-2)k+\varepsilon} \mid k=0,1,2,...,4n-1 \}=\{ba^{4t+\epsilon} \mid t=0,1,2,...,n-1\}$, it is enough to show that for each $t \in \{0,1,2,..., n-1\}$, there is $k \in \{0,1,2,...,4n-1 \}$ such that $[(2n-2)k+\varepsilon]\equiv [4t+\varepsilon] \hbox{ mod } 4n$.  This is equivalent to checking that the linear Diophantine equation $(n_{0}-1)k+ns=t$ has a solution.  This is obvious because the $\gcd(n_{0}-1,n)=1$ (since $n=2n_{0}-1$) always divides $t$.
\end{proof}

To find the character table for $SD_{8n}$, we first recollect the main results for computing character table from \cite{JS}.
\begin{prop}\label{prop character} \cite{JS}
Let $V$ be a complex vector space of dimension $n$ and $G$ be a finite group.
If $\chi$ is the character of a representation $\rho$ ($\rho:G\rightarrow GL(V)$) of degree n i.e. $\chi_{\rho}(s)=Tr(\rho(s))$ for each $s\in G$, we have:
\begin{description}
  \item[$(1)$] $\chi(1)=n$, degree of $\rho$.
  \item[$(2)$] $\chi(s^{-1})= \overline{\chi(s)}$, conjugate of complex number, for all $s\in G$.
  \item[$(3)$] $\chi(tst^{-1})= \chi(s)$ for all $s\in G$.
  \item[$(4)$] If $\phi$ is the character of a representation $V$, then $(\phi,\phi)$ is a positive integer and we have
       $(\phi,\phi)=1 $ if and only if $V$ is irreducible, where
       $$(\phi,\phi)=\frac{1}{|G|}\sum_{s\in G}\phi(s)\overline{\phi(s)}. $$
  \item[$(5)$] Two representations with the same character are isomorphic.\\
   ( \textbf{Note:} $ \rho\cong\rho'\Leftrightarrow TR_{s}=R'_{s}T$ for some invertible matrix $T$ and for all $s\in G$, where $R_{s}$ and $R'_{s}$ are the representation matrixes of $\rho(s)$ and $\rho'(s)$ respectively.)
  \item[$(6)$] The number of irreducible representations of $G$ (up to isomorphism) is equal to the number of conjugacy classes of $G$.
  \item[$(7)$] The degree of the irreducible representation of $G$ divide the order of $G$.  Furthermore, it also divides $ (G:C)$ where $C$ is the centre of $G$.
  \item[$(8)$] If the irreducible characters of $G$ are $\chi_{1},\chi_{2},...,\chi_{h}$ then $|G|=\sum_{i = 1}^h n_{i}^2 $ where $n_{i}=\chi_{i}(1)$ and if $s\in G$ is different from 1, then we have $\sum_{i = 1}^h n_{i}\chi_{i}(s)=0$.
\end{description}
\end{prop}

Since the numbers of conjugacy classes of $SD_{8n}$ are different for even $n$ ($2n+3$ classes) or odd $n$ ($2n+6$ classes), we consider the character tables separately.

We first consider linear representations for even $n$.  By the relations $a^{4n}=1$ and $b^{2}=1$ and Proposition \ref{prop character} (4), we see that $\chi_{0}, \chi_{1}, \chi_{2}, \chi_{3}$ defined by $\chi_{0}(a)=1$, $\chi_{0}(b)=1$, $\chi_{1}(a)=1$, $\chi_{1}(b)=-1$, $\chi_{2}(a)=-1$, $\chi_{2}(b)=1$, and $\chi_{3}(a)=-1$, $\chi_{3}(b)=-1$ are irreducible linear representations.  Note also for even $n$ that there are no linear representation $\chi$ such that $\chi(a)=i:=\sqrt{-1}$, since $\chi(a)$  must be equal to $\chi(a^{2n-1})$ (by Proposition \ref{prop character} (3)).

We can conclude from ($7$) and ($8$) in Proposition \ref{prop character} that the groups $SD_{8n}$, for even $n$, must contain $2n-1$ two dimensional irreducible representations.

Now, we define two dimensional representations, for each natural number $h$ and $\omega=e^{\frac{i\pi}{2n}}$;
\begin{equation}\label{two dim rep}
    \rho^{h}(a^{r})=\left(
                        \begin{array}{cc}
                       \omega^{hr} & 0 \\
                       0 & \omega^{(2n-1)hr}
                     \end{array}
    \right) \hbox{ and } \rho^{h}(ba^{r})=\left(\begin{array}{cc}
                                       0 & \omega^{(2n-1)hr} \\
                                       \omega^{hr} & 0
                                     \end{array}
    \right),
\end{equation}
for each $r \in \{1,2,...,4n \}$.  It is easy to check that $\rho^{h}$ is a representation for any $h$.  Since $\omega^{4n}=1$, it is easy to see (by using (5) in Proposition \ref{prop character}) that $\rho^{h}\cong\rho^{(2n-1)h}$.  Thus, it suffices to consider $h \in C^{even}$.
\begin{lem}\label{number of irr re}
For even $n$, all $\rho^{h}$, $h\in C_{*}^{even}$, are irreducible representations.
\end{lem}
\begin{proof} Let $h\in C^{even}$.  We consider (by (5) in Proposition \ref{prop character}) the character of $\rho^{h}$, $\chi_{\rho^{h}}$,
\begin{eqnarray*}
  (\chi_{\rho^{h}},\chi_{\rho^{h}}) &=& \frac{1}{8n}\sum_{x\in SD_{8n}}\chi_{\rho^{h}}(x)\overline{\chi_{\rho^{h}}(x)} \\
   &=& \frac{1}{8n}\sum_{r=1}^{4n}\chi_{\rho^{h}}(a^{r})\overline{\chi_{\rho^{h}}(a^{r})}, \hbox{(since $\chi_{\rho^{h}}(ba^{r})=0, \forall r\in\{1,2,3,...,4n\}$)} \\
   &=& \frac{1}{8n}\sum_{r=1}^{4n} (w^{hr}+w^{(2n-1)hr})\overline{(w^{hr}+w^{(2n-1)hr})} \\
   &=& \frac{1}{8n}\sum_{r=1}^{4n} [2+(w^{(2n-2)hr}+w^{-(2n-2)hr})].
\end{eqnarray*}
Thus, $$(\chi_{\rho^{h}},\chi_{\rho^{h}}) = \frac{1}{8n}\sum_{r=1}^{4n} \left[2+2\cos\frac{(2n-2)hr\pi}{2n}\right].$$
By using the formula $ 1+2\cos x+ 2\cos 2x+ 2\cos 3x+...+2\cos rx=\frac{\sin ((r+\frac{1}{2})x)}{\sin \frac{x}{2}}$ for $\sin \frac{x}{2}\neq 0$, with $x=\frac{(2n-2)h\pi}{2n}$ and formula $\sin (2\pi +\theta)=\sin(\theta)$, (here, the points $h$ that make $\sin\frac{(2n-2)h\pi}{4n} = 0$ are only $0,2n$ since $n$ is even),  we now conclude that, for $h\in C^{even}_{*}$, $$\frac{1}{8n}\sum_{r=1}^{4n} 2\cos\frac{(2n-2)hr\pi}{2n} = 0$$ i.e. $(\chi_{\rho^{h}},\chi_{\rho^{h}}) = 1$, $\forall h\in C^{even}_{*}$  . Also, it is easy to see that $(\chi_{\rho^{0}},\chi_{\rho^{o}})$ and $(\chi_{\rho^{2n}},\chi_{\rho^{2n}})$ are not 1, then $\rho^{0}$ and $\rho^{2n}$ are not irreducible which completes the proof.
\end{proof}
We now tabulate the character table for $SD_{8n}$, where $n$ is even as follows.
\begin{center}
\begin{tabular}{|c|c|c|c|c|}
  \hline
  \hline
   Conjugacy classes,& $[a^{r}];$ & $[a^{r}];$ & $[b]$ & $[ba]$ \\
   Characters& $r \in C_{1}$ & $r \in C^{\dag}_{odd}$ &  &  \\
   \hline
   \hline
  $\chi_{0}$ & 1 & 1 & 1 & 1 \\
  \hline
 $\chi_{1}$ & 1 & 1 & -1 & -1 \\
  \hline
 $\chi_{2}$ & 1 & -1 & 1 & -1 \\
  \hline
 $\chi_{3}$ & 1 & -1 & -1 & 1 \\
  \hline
  &&&&\\
  $\varsigma_{h}$, where & $2\cos(\frac{hr\pi}{2n})$ & $2\cos(\frac{hr\pi}{2n})$ & 0 & 0 \\
  $h \in  C^{\dag}_{even}$ &  &  &  &  \\
  \hline
   &&&&\\
   $\psi_{h}$, where & $2\cos(\frac{hr\pi}{2n})$ & $2i\sin(\frac{hr\pi}{2n})$ & 0 & 0 \\
  $h \in  C^{\dag}_{odd}$ &  &  &  &  \\
  \hline
  \hline
\end{tabular}\\
\end{center}
\begin{center}
    \textbf{Table I} The character table for $SD_{8n}$, where $n$ is even.
\end{center}

Now, the character table for $SD_{8n}$, where $n$ is odd, can be obtained in the same way as for even $n$.  There are eight irreducible linear representations; namely, four of them are the same as in the even case and the other four are $\chi_{4}, \chi_{5}, \chi_{6}, \chi_{7}$ defined by $\chi_{4}(a)=i$, $\chi_{4}(b)=1$, $\chi_{5}(a)=i$, $\chi_{5}(b)=-1$, $\chi_{6}(a)=-i$, $\chi_{6}(b)=1$, and $\chi_{7}(a)=-i$, $\chi_{7}(b)=-1$.

There are $2n-2$ irreducible two dimension representations defined as (\ref{two dim rep}).
\begin{lem}
For odd $n$, all $\rho^{h}$, $h\in C_{*}^{odd}$, are irreducible representations.
\end{lem}
\begin{proof} The proof proceeds as in Lemma \ref{number of irr re}; i.e. we check whether $ (\chi_{\rho^{h}},\chi_{\rho^{h}})$ is zero or not.  We see that the points $h$ that make $\sin\frac{(2n-2)h\pi}{4n} = 0$ are $0,n,2n$ and $3n$ since $n$ is odd and again it is easy to see that $\rho^{0}, \rho^{n}, \rho^{2n}$ and $\rho^{3n}$ are not irreducible representations.
\end{proof}

We now tabulate the character table for $SD_{8n}$, where $n$ is odd as follows.
\begin{center}
\begin{tabular}{|c|c|c|c|c|c|c|}
  \hline
  \hline
   Conjugacy classes,& $[a^{r}];$ & $[a^{r}];$ & $[b]$ & $[ba]$ & $[ba^{2}]$ & $[ba^{3}]$ \\
   Characters& $r \in C_{1}$ & $r \in C_{2,3}^{odd}$ &  &  && \\
   \hline
   \hline
  $\chi_{0}$ & 1 & 1 & 1 & 1 &1&1\\
  \hline
 $\chi_{1}$ & 1 & 1 & -1 & -1 &-1&-1 \\
  \hline
 $\chi_{2}$ & 1 & -1 & 1 & -1 &1&-1\\
  \hline
 $\chi_{3}$ & 1 & -1 & -1 & 1 &-1&1\\
  \hline
  $\chi_{4}$ & $(-1)^{\frac{r}{2}}$ & $i^{r}$ & 1 & $i$ &-1&$-i$\\
  \hline
 $\chi_{5}$ & $(-1)^{\frac{r}{2}}$& $i^{r}$ & -1 & $-i$ &1&$i$ \\
  \hline
 $\chi_{6}$ &$(-1)^{\frac{r}{2}}$ & $(-i)^{r}$ & 1 & $-i$ &-1&$i$\\
  \hline
 $\chi_{7}$ &$(-1)^{\frac{r}{2}}$ & $(-i)^{r}$ & -1 & $i$ &1&$-i$\\
  \hline
  &&&&&&\\
  $\varsigma_{h}$, where & $2\cos(\frac{hr\pi}{2n})$ & $2\cos(\frac{hr\pi}{2n})$ & 0 & 0 &0&0\\
  $h \in C^{\dag}_{even}$ &  &  &  &  &&\\
  \hline
   &&&&&&\\
   $\psi_{h}$, where & $2\cos(\frac{hr\pi}{2n})$ & $2i\sin(\frac{hr\pi}{2n})$ & 0 & 0 &0&0\\
  $h \in C^{odd}_{2,3}\setminus \{n,3n \}$ &  &  &  &  &&\\
  \hline
  \hline
\end{tabular}\\
\end{center}
\begin{center}
    \textbf{Table II} The character table for $SD_{8n}$, where $n$ is odd.
\end{center}

To count the number of disjoint cycles of elements in semidihedral group, it is necessary to see the explicit embedding.
\begin{prop}\label{embedding} For semidihedral groups, $ SD_{8n}=<a,b \mid a^{4n}=b^{2}=1,bab=a^{2n-1}>, n\geq2 $,
the embedding $T:SD_{8n}\hookrightarrow S_{4n}$ is given by $T(a)(t):=\overline{t+1}$ and $T(b)(t):=\overline{(2n-1)t}$, where $\overline{m}$ is the remainder of $m$ divided by $4n$.  Explicitly,
\begin{eqnarray*}
  T(a) &=& (\begin{array}{ccccc}
              1 & 2 & 3 & \cdots & 4n
            \end{array}
  ) \\
  T(b) &=& \prod_{i \in C^{even}_{*}}(\begin{array}{cc}
                                        i & \overline{(2n-1)i}
                                      \end{array})
    \hbox{ if n is even,}\\
  T(b) &=& \prod_{i \in C^{odd}_{*}}( \begin{array}{cc}
                                        i & \overline{(2n-1)i}
                                      \end{array}) \hbox{ if n is odd,}
\end{eqnarray*}
\end{prop}
\begin{proof} Since the lengths of each cycle in $T(a)$ and $T(b)$ are $4n$ and $2$ respectively, $T(a^{4n})=e=(T(a))^{4n}$ and $T(b^{2})=e=(T(b))^{2}$. Also
\begin{eqnarray*}
T(bab)(t)&\overset{4n}{\equiv}&(2n-1)((2n-1)t+1)  \\
    &\overset{4n}{\equiv}& t+(2n-1) \\
     &\overset{4n}{\equiv}& T(a^{2n-1})(t).
\end{eqnarray*}
\end{proof}

Now, by using the embedding in Proposition \ref{embedding}, we obtain the following result.
\begin{prop} \label{c(g)}  Let $c(g)$ be the number of cycles, including cycles of length one, in the disjoint cycle factorization of $T(g)$.  Then, for even $n$,
$$\begin{array}{ccc}
    c(a^{r})=\gcd(4n,r), & c(a^{2r-1}b)=n, & c(a^{2r}b)=2n+1 .
  \end{array}
  $$
For odd $n$;
$$\begin{array}{ccc}
   c(a^{r})=\gcd(4n,r), & c(a^{2r-1}b)=n, & c(a^{2r}b)=\left\{
                                                                               \begin{array}{ll}
                                                                                 2n+2, & \hbox{if $r$ is even;} \\
                                                                                 2n, & \hbox{if $r$ is odd.}
                                                                               \end{array}
                                                                             \right.
  \end{array}
  $$
\end{prop}
\begin{proof}  Notice that $<a^{r}>$ has as many elements as the order of $a^{r}$.  Thus the number of elements of $<a^{r}>$  is $4n/\gcd(4n,r)$ and hence we have $c(a^{r})=\gcd(4n,r)$ and $c(1)=4n$.  For $c(a^{r}b)$, we use the action of $T(a^{r})$ sending $t$ to $\overline{t+r}$ and the action of $T(b)$ sending $t$ to $\overline{t(2n-1)}$. We thus compute that for $t \in \{ 1,2,...,4n\}$,
$$ T(a^{r}b)_{(t)}=\left\{
                     \begin{array}{ll}
                       (\begin{array}{cc}
                         t & \overline{2n-t+r}
                       \end{array})
                       , & \hbox{if r is even and t is odd;} \\
                       (\begin{array}{cc}
                          t & \overline{r-t}
                        \end{array}
                       ), & \hbox{if r is even and t is even.}
                     \end{array}
                   \right.
 $$
Thus these cycles have length one if $t = \overline{2n-t+r} $ or $t = \overline{r-t} $.  That is, for even $r$, we  calculate and tabulate the cycles $T(a^{r}b)_{(t)}$ (here, $T(s)_{(t)}$ is a cycle in the decomposition of $T(s)$ starting with $t$) of length one as
\begin{center}
   \begin{tabular}{|c|l|l|l|l|}
       \hline
       \hline
      \textbf{ r even} & \textbf{n even; t odd} & \textbf{n even; t even} & \textbf{n odd; t odd} & \textbf{n odd; t even} \\
       \hline
       &&&&\\
       $\gcd(r,4)=4$ & No such $t$ & $t=\overline{2n+\frac{r}{2}}$ and & $t=\overline{n+\frac{r}{2}}$ and & $t=\overline{2n+\frac{r}{2}}$ and \\
        & & $t=\overline{\frac{r}{2}}$ & $t=\overline{3n+\frac{r}{2}}$& $t=\overline{\frac{r}{2}}$ \\
         &&&&\\
        \hline
        &&&&\\
       $\gcd(r,4)=2$ & $t=\overline{n+\frac{r}{2}}$ and &  No such $t$&  No such $t$ &  No such $t$ \\
        & $t=\overline{3n+\frac{r}{2}}$ &  &  & \\
         &&&&\\
       \hline
       \hline
     \end{tabular}
\end{center}
Hence, for even $r$ and even $n$, we see that $T(a^{r}b)$ contains only two cycles of length $1$ and then in this case $c(a^{r}b)=\frac{4n-2}{2}+2=2n+1$.  Also, for even $r$ such that $\gcd(r,4)=4$ and $n$ odd, we see that $T(a^{r}b)$ contains only four cycles of length $1$ and then in this case $c(a^{r}b)=\frac{4n-4}{2}+4=2n+2$.  For even $r$ such $\gcd(r,4)=2$ and $n$ odd, we see that $T(a^{r}b)$ contains no cycle of length $1$ and then in this case $c(a^{r}b)=\frac{4n}{2}=2n$.

For odd $r$  and $t \in \{ 1,2,...,4n\}$, we compute that
$$ T(a^{r}b)_{(t)}=\left\{
                     \begin{array}{ll}
                       (\begin{array}{cccc}
                         t & \overline{2n+r-t} & \overline{2n+t} & \overline{r-t}
                       \end{array})
                       , & \hbox{if  t is odd;} \\
                       (\begin{array}{cccc}
                          t & \overline{r-t} & \overline{2n+t} & \overline{2n+r-t}
                        \end{array}
                       ), & \hbox{if t is even.}
                     \end{array}
                   \right.
 $$
 Since these numbers are not congruent in modulo $4n$, all cycles in the factors of $T(a^{r}b)$ in this case are of length $4$ and hence, for odd $r$ and any $n$, $c(a^{r}b)=\frac{4n}{4}=n$.
\end{proof}

In the following theorems, we find the dimensions of the symmetry classes of tensors associated with the group $SD_{8n}$.\\

\begin{thm}\label{thmdim1}  Let $G=SD_{8n}$, $n$ even, and let $V$ be an $m$-dimensional inner product space. Let $\overline{S}=\{ \overline{s(2n-1)} \mid s \in S, \hbox{ where $\overline{r}$ is the remainder of $r$ divided by $4n$}\}$, $\varsigma_{h}$ for $h\in  C^{\dag}_{even}$ and $\psi_{h'}$ for $h'\in C^{\dag}_{odd}$, then we have
\begin{eqnarray*}
 \dim V^{4n}_{\chi_0}(G)&=& \frac{1}{8n}\left[2nm^{n}+2nm^{2n+1}+\sum_{k\in \{0,1,2,...,4n-1 \}}m^{\gcd(4n,k)}  \right] \\
 \dim V^{4n}_{\chi_1}(G)&=&\frac{1}{8n}\left[-2nm^{n}-2nm^{2n+1}+\sum_{k\in \{0,1,2,...,4n-1 \}}m^{\gcd(4n,k)}   \right] \\
\dim V^{4n}_{\chi_2}(G)&=& \frac{1}{8n}\left[m^{4n}+m^{2n}+2\sum_{k\in  C^{\dag}_{even}}m^{\gcd(4n,k)}-\sum_{k \in C^{\dag}_{odd} \cup \overline{C^{\dag}_{odd}} }m^{\gcd(4n,k)}+ 2nm^{2n+1}-2nm^{n}  \right]\\
\dim V^{4n}_{\chi_3}(G)&=& \frac{1}{8n}\left[m^{4n}+m^{2n}+2\sum_{k\in  C^{\dag}_{even}}m^{\gcd(4n,k)}-\sum_{k \in C^{\dag}_{odd} \cup \overline{C^{\dag}_{odd}} }m^{\gcd(4n,k)} - 2nm^{2n+1}+2nm^{n}  \right]\\
\dim V^{4n}_{\varsigma_{h}}(G)&=& \frac{1}{2n}\left[ \sum_{k\in \{ 0,1,2,...,4n-1\}}m^{\gcd(4n,k)} \cos (\frac{hk\pi}{2n}) \right] \\
\dim V^{4n}_{\psi_{h'}}(G)&=&\frac{1}{4n}\left[ m^{4n}-m^{2n}+4\sum_{k\in  C^{\dag}_{even}}m^{\gcd(4n,k)} \cos (\frac{h'k\pi}{2n}) \right]
\end{eqnarray*}
\end{thm}
\begin{proof} This follows by (\ref{important1}) and Table I together with Proposition \ref{c(g)}.  For $\dim V^{4n}_{\chi_2}(G)$, $\dim V^{4n}_{\chi_3}(G)$ and $\dim V^{4n}_{\psi_{h'}}(G)$, we have used  the fact that $\overline{(2n-1)k}=4n-k$ for $k \in C_{1}$ and $\gcd(4n,4n-k)=\gcd(4n,k)$.  Note also, for $\dim V^{4n}_{\psi_{h'}}(G)$,  we do not need to compute $\psi_{h'}(g)$ for $g \in C^{\dag}_{odd}$ since they are imaginary number (the dimension must be integer); i.e.,
\begin{equation}\label{isin}
    2i\sum_{k\in C^{\dag}_{odd}\cup \overline{C^{\dag}_{odd}} }m^{\gcd(4n,k)} \sin (\frac{h'k\pi}{2n})=0.
\end{equation}
\end{proof}
\begin{thm}\label{thmdim 2}  Let $G=SD_{8n}$, $n$ odd, and let $V$ be an $m$-dimensional inner product space. Let $\overline{S}=\{ \overline{s(2n-1)} \mid s \in S, \hbox{ where $\overline{r}$ is the remainder of r divided by $4n$}\}$, $\varsigma_{h}$ for $h\in C^{\dag}_{even}$ and $\psi_{h'}$ for $h' \in C^{odd}_{2,3}\setminus\{n,3n\}$, then we have
\begin{eqnarray*}
 \dim V^{4n}_{\chi_0}(G)&=& \frac{1}{8n}\left[2nm^{n}+nm^{2n}+nm^{2n+2}+ \sum_{k\in \{ 0,1,2,...,4n-1\}}m^{\gcd(4n,k)} \right] \\
\dim V^{4n}_{\chi_1}(G)&=& \frac{1}{8n}\left[-2nm^{n}-nm^{2n}-nm^{2n+2}+ \sum_{k\in \{ 0,1,2,...,4n-1\}}m^{\gcd(4n,k)}  \right] \\
\dim V^{4n}_{\chi_2}(G)&=& \frac{1}{8n}\left[ m^{4n}+m^{2n}-2m^{n}+2\sum_{k\in C^{\dag}_{even}}m^{\gcd(4n,k)}-\sum_{k \in C^{\dag}_{odd}\cup \overline{C^{\dag}_{odd}}}m^{\gcd(4n,k)} -2 nm^{n}\right.\\
&+&  \big. nm^{2n}+nm^{2n+2}  \big]\\
\dim V^{4n}_{\chi_3}(G)&=& \frac{1}{8n}\left[ m^{4n}+m^{2n}-2m^{n}+2\sum_{k\in C^{\dag}_{even}}m^{\gcd(4n,k)}-\sum_{k \in C^{\dag}_{odd}\cup \overline{C^{\dag}_{odd}}}m^{\gcd(4n,k)}2 nm^{n}\right. \\
&-&\big. nm^{2n}-nm^{2n+2}  \big]\\
\dim V^{4n}_{\chi_4}(G)&=& \frac{1}{8n}\left[ m^{4n}-m^{2n}+\sum_{k\in C^{\dag}_{even}\cup \overline{C^{\dag}_{even}}}(-1)^{\frac{k}{2}}m^{\gcd(4n,k)} + nm^{2n+2} - nm^{2n}  \right]\\
\dim V^{4n}_{\chi_5}(G)&=&\frac{1}{8n}\left[m^{4n}-m^{2n}+\sum_{k\in C^{\dag}_{even}\cup \overline{C^{\dag}_{even}}}(-1)^{\frac{k}{2}}m^{\gcd(4n,k)} - nm^{2n+2} + nm^{2n}   \right] \\
\dim V^{4n}_{\chi_6}(G)&=&\frac{1}{8n}\left[ m^{4n}-m^{2n}+\sum_{k\in C^{\dag}_{even}\cup \overline{C^{\dag}_{even}}}(-1)^{\frac{k}{2}}m^{\gcd(4n,k)}+ nm^{2n+2} - nm^{2n}  \right] \\
\dim V^{4n}_{\chi_7}(G)&=&\frac{1}{8n}\left[m^{4n}-m^{2n}+\sum_{k\in C^{\dag}_{even}\cup \overline{C^{\dag}_{even}}}(-1)^{\frac{k}{2}}m^{\gcd(4n,k)} - nm^{2n+2} + nm^{2n} \right]\\
\dim V^{4n}_{\varsigma_{h}}(G)&=& \frac{1}{2n}\left[ \sum_{k\in \{ 0,1,2,...,4n-1\}}m^{\gcd(4n,k)} \cos (\frac{hk\pi}{2n}) \right] \\
\dim V^{4n}_{\psi_{h'}}(G)&=&\frac{1}{4n}\left[ m^{4n}-m^{2n}+4\sum_{k\in C^{\dag}_{even}}m^{\gcd(4n,k)} \cos (\frac{h'k\pi}{2n}) \right]
\end{eqnarray*}
\end{thm}
\begin{proof} The proof is similar to the proof of Theorem \ref{thmdim1}.
\end{proof}

\section{\bf  On the existence of an orthogonal basis for the symmetry classes of tensors associated with $SD_{8n}$}

In this section we study the existence of an orthogonal basis for the symmetry classes of tensors associated with $SD_{8n}$.  According to Section 2 and Table I,II, we have four (in case of even $n$) or eight (in case of odd $n$) irreducible characters of degree one  and $n-1$ characters $\varsigma_{h}$, $h\in  C^{\dag}_{even}$ of degree $2$ ( in case of both even and odd $n$), and a further $n$ characters $\psi_{h'}$, $h'\in  C^{\dag}_{odd}$ of degree $2$ (in case of even $n$) or $n-1$ characters $\psi_{h'}$, $h' \in  C^{odd}_{2,3}\setminus\{n,3n \}$ of degree $2$ (in case of odd $n$). As we explained in the introduction if $\chi$ is a linear character of $G$ then the symmetry class of tensors associated with $G$ and $\chi$ has an orthogonal basis. Therefore we will concentrate on  non-linear irreducible complex characters of $SD_{8n}$, i.e. the characters $\varsigma_{h}$, $h\in  C^{\dag}_{even}$ and $\psi_{h'}$ for $h' \in  C^{\dag}_{odd} $ or $h' \in  C^{odd}_{2,3}\setminus\{n,3n \} $. \\

 It turns out that $V_{\psi_{h'}}(SD_{8n})$ does not have an orthogonal basis for any odd or even $n$ (see Theorem \ref{thm odd}). The same result is obtained for $V_{\varsigma_{h}}(SD_{8n})$ if $n$ is odd (see Corollary \ref{cor odd}).  However,  there is an orthogonal basis for $V_{\varsigma_{h}}(SD_{8n})$ if $h\in  C^{\dag}_{even}$ and the condition $\nu_{2}(\frac{h}{2n})<0$ holds (see Theorem \ref{main thm1}).

\begin{rem}
 Let $\nu_2$ be the $2$-adic valuation, that is $\nu(\frac{2^{k}m}{n})=k$ for $m$ and $n$ odd. Then, the condition $\nu_2(\frac{h}{2n})< 0$ means that every power of $2$ that divides $h$ also divides $n$.
\end{rem}
\begin{lem}
\label{l1}
 Let $G:=SD_{8n}$ and $H$ be a subgroup of $G$. Then there is a natural number $r$, $0\leq r < 4n$ such that $H = \langle a^r \rangle $ or $\langle a^r\rangle \lneqq H$ and $H\cap \langle a \rangle =\langle a^r \rangle$. In the second case we have $|H|\geq 2 |\langle a^r \rangle|$.
\end{lem}
\begin{proof} It is straightforward.
\end{proof}
\begin{lem}\label{l2} Suppose $\varsigma=\varsigma_{h}$.  If $r$ is defined by $ G_\alpha \cap \langle a \rangle =\langle a^r \rangle$ and $l=\frac{4n}{gcd(4n,r)}$, then we have
  $$
\sum_{g\in G_\alpha}\varsigma (g)= \left\{\begin{array}{ll} 2l,\hspace{0.7cm} \text{if} \hspace{0.7cm}rh \equiv0 ~ (mod ~4n)  \\
0,\hspace{0.7cm} \text{if} \hspace{0.9cm} rh \not \equiv 0 ~ (mod ~4n)
\end{array}\right.$$
and for $\alpha \in \overline{\Delta }$, we have $rh \equiv 0 ~(mod ~4n )$.

\end{lem}

\begin{proof}
Since $G_\alpha$ is a subgroup of $G$, using Lemma \ref{l1} there is a natural number $r$, $0\leq r < 4n$ such that $G_\alpha = \langle a^r \rangle $  or $\langle a^r \rangle < G_\alpha$. Using Table I, $\varsigma$ vanishes outside $\langle a \rangle$, therefore
 $$
 \sum_{g\in G_\alpha}\varsigma (g)= \sum_{t=1}^{l}\varsigma (a^{tr})=2\sum_{t=1}^{l}\cos(\frac{trh \pi}{2n})= \left\{\begin{array}{ll} 2l,~  rh \equiv0 ~ (mod ~4n)  \\
0,~  rh \not \equiv 0 ~~(mod ~4n).
\end{array}\right.
 $$
 Also if $rh \not \equiv 0 ~ (mod ~4n)$, then  $\sum_{g\in G_\alpha}\varsigma (g)= 0$ which shows $\alpha \notin \overline{\Delta }$.
\end{proof}

\begin{lem}
\label{l5}
Let $1 \leq h< 2n$ and $\nu_2$ is the $2$-adic valuation. Then there exist $t_1, t_2$,~ $0\leq t_1, t_2 <4n$
such that $\cos(\frac{(t_1-t_2)h\pi }{2n} )=0$ if and only if $\nu_2(\frac{h}{2n})< 0$.
\end{lem}

\begin{thm}\label{main thm1}
\label{t1}
Let $G =SD_{8n}$ be a subgroup of $S_{4n}$, denote $\varsigma=\varsigma_{h}$ for $h \in C^{\dag}_{even}$, and assume $d= \dim V\geq 2 $. Then $V_{\varsigma}(G)$ has an orthogonal $\ast$-basis  if and only $\nu_2(\frac{h}{2n})< 0$ .
\end{thm}
\begin{proof}
It is enough to prove that for any $\alpha \in \overline{\Delta} $ the orbital subspace $V^{*}_\alpha$
has orthogonal $\ast$-basis  if $\nu_2(\frac{h}{2n})< 0$. Let $\nu_2(\frac{h}{2n})< 0$ and
assume $\alpha \in \overline{\Delta} $. By Lemma \ref{l1}, $G_\alpha = \langle a^r \rangle $  or $\langle a^r \rangle < G_\alpha$. Let $l=\frac{4n}{gcd(4n,r)}$. Now we consider two cases.\\

\textbf{Case 1}. If $\langle a^r \rangle < G_\alpha$, then by Lemma \ref{l1} we establish $|G_\alpha| \geq 2l$ where
$$\langle a^r \rangle= \langle a \rangle \cap G_\alpha = \{a^r,a^{2r},...,a^{lr}=1  \}.$$
By (\ref{important3}), $|G_\alpha| \geq 2l$ and Lemma \ref{l2}, we have
\begin{equation}
\nonumber
\dim V^{*}_\alpha =\frac{\varsigma(1)}{|G_\alpha|}\sum_{\sigma \in G_\alpha}\varsigma(\sigma)\leq \frac{2}{2l}(2l)=2.
\end{equation}
 If $\dim V^{*}_\alpha=1$, then it is obvious that we have an orthogonal $\ast$-basis. Let us consider $\dim V^{*}_\alpha=2$. Set $\sigma_1=a^j,\sigma_2=a^i $. Then
$$
\sigma_2G_\alpha{\sigma_1}^{-1}\cap \langle a \rangle= \{ a^{r+i-j},..., a^{lr+i-j}\}.
$$
Hence if $\sigma_1=a^j,\sigma_2=a^i$, by (\ref{important2}), we have

\begin{eqnarray*}
\langle e^{*}_{\sigma_1.\alpha},e^{*}_{\sigma_2.\alpha}\rangle &=&\frac{\varsigma(1)}{|G|}\sum_{x \in \sigma_2 G_\alpha \sigma^{-1}_1}\varsigma(x)= \frac{2}{8n}\sum_{t=1}^{l}\varsigma(a^{tr+i-j}) \\
&=& \frac{4}{8n}\sum_{t=1}^{l} \cos \frac{(tr+i-j)h \pi}{2n}\\
&=& \frac{1}{2n}\sum_{t=1}^{l} \cos ( \frac{trh \pi}{2n}+\frac{(i-j)h \pi}{2n}) \\
&=&\frac{1}{2n}\sum_{t=1}^{l} \cos ( \frac{(i-j)h \pi}{2n})= \frac{l}{2n}\cos ( \frac{(i-j)h \pi}{2n})\hspace{0.5cm}(3.1)
\end{eqnarray*}
where penultimate equality is due to application of  Lemma \ref{l2}. Using Lemma \ref{l5}, there exist $i$ and $j$ such that
$$
\langle e^{*}_{a^j.\alpha},e^{*}_{a^i.\alpha}\rangle=0
$$
which means that $\{ e^{*}_{\sigma_1.\alpha},e^{*}_{\sigma_2.\alpha}\} $ is an orthogonal $\ast$-basis for $V^{*}_\alpha$.\\

\textbf{Case 2.} If $G_\alpha = \langle a^r \rangle=\{ a^r,a^{2r},...,a^{lr}=1 \}$, then by (\ref{important3}) and Lemma \ref{l2},
\begin{equation}
\nonumber
\dim V^{*}_\alpha =\frac{\varsigma(1)}{|G_\alpha|}\sum_{\sigma \in G_\alpha}\varsigma(\sigma)= \frac{2}{l} (2l)=4.
\end{equation}
 For any $\sigma_1,\sigma_2 \in G$, we have
$$
\sigma_2 G_\alpha {\sigma_1~}^{-1}= \left\{\begin{array}{ll} \{ a^{r+i-j},a^{2r+i-j},...,a^{lr+i-j}\} , \hspace{4.0cm}\text{if} \hspace{0.7cm}\sigma_1=a^j,\sigma_2=a^i \\
\{ a^{r+i+j(1-2n)}b,a^{2r+i+j(1-2n)}b,...,a^{lr+i+j(1-2n)}b \} ,\hspace{0.7cm}\text{if} \hspace{0.7cm}\sigma_1=a^jb,\sigma_2=a^i \\
\{  a^{(1-2n)r+i-j},a^{2r(1-2n)+i-j},...,a^{lr(1-2n)+i-j} \} ,\hspace{1.2cm}\text{if} \hspace{0.8cm}\sigma_1=a^jb,\sigma_2=a^ib
\end{array}\right.
$$
If $\sigma_1=a^j,\sigma_2=a^i$, by (3.1) we have
$$\langle e^{*}_{\sigma_1.\alpha},e^{*}_{\sigma_2.\alpha}\rangle  = \frac{l}{2n}\cos ( \frac{(i-j)h \pi}{2n})$$
If $\sigma_1=a^jb,\sigma_2=a^i$, we have
 $$
 \langle e^{*}_{\sigma_1.\alpha},e^{*}_{\sigma_2.\alpha}\rangle =0
 $$
and for $\sigma_1=a^jb,\sigma_2=a^ib $, we have

\begin{eqnarray*}
\langle e^{*}_{\sigma_1.\alpha},e^{*}_{\sigma_2.\alpha}\rangle&=& \frac{\varsigma(1)}{|G|}\sum_{x \in \sigma_2 G_\gamma \sigma^{-1}_1}\varsigma(x)= \frac{2}{8n}\sum_{t=1}^{l}\varsigma(a^{tr(1-2n)+i-j})\\
&=&\frac{4}{8n}\sum_{t=1}^{l} \cos \frac{(tr(1-2n)+i-j)h \pi}{2n}\\
&=& \frac{1}{2n}\sum_{t=1}^{l} \cos ( \frac{trh \pi}{2n}+\frac{(i-j)h \pi}{2n}-trh\pi) \\
&=&\frac{1}{2n}\sum_{t=1}^{l} \cos ( \frac{(i-j)h \pi}{2n})= \frac{l}{2n}\cos ( \frac{(i-j)h \pi}{2n})
\end{eqnarray*}
where penultimate equality is due to application of  Lemma \ref{l2}. Therefore
 $$
 \langle e^{*}_{\sigma_1.\alpha},e^{*}_{\sigma_2.\alpha}\rangle= \left\{\begin{array}{ll} \frac{l}{2n}\cos ( \frac{(i-j)h \pi}{2n}),\hspace{0.5cm} \sigma_1=a^j,\sigma_2=a^i \\
0,  \hspace{2.7cm} \sigma_1=a^jb,\sigma_2=a^i \\
 \frac{l}{2n}\cos ( \frac{(i-j)h \pi}{2n}),\hspace{0.5cm}\sigma_1=a^jb,\sigma_2=a^ib
\end{array}\right.
$$
Applying Lemma \ref{l5}, if $\nu_2(\frac{h}{2n})< 0$, there exist $t_1, t_2$,~ $0\leq t_1, t_2 <4n$
such that $\cos(\frac{(t_1-t_2)h\pi }{2n} )=0$. Put
$$
S= \{ a^{t_1}.\alpha,a^{t_2}.\alpha,a^{t_1}b.\alpha,a^{t_2}b.\alpha \} \subseteq \Gamma^{m}_n.
$$
Then for every $\alpha,\beta \in S$ and $\alpha \neq \beta$ we have
$$
\langle e^{*}_{\alpha},e^{*}_{\beta}\rangle=0
$$
But  $\dim V^{*}_\alpha = 4$; hence $\{e^{*}_{\xi} | \xi \in S \} $ is an orthogonal $\ast$-basis for $V^{*}_\alpha$.\\

Conversely, assume that $V_{\varsigma}(G)$ has an orthogonal basis of decomposable symmetrized tensors. Then since $V_\varsigma(G)= \bigoplus_{\alpha \in \overline{\Delta}} V^{*}_\alpha$ for all $\alpha \in \overline{\Delta}$, the orbital subspace $V^{*}_\alpha$ has an orthogonal basis of decomposable symmetrized tensors. Using [17, p. 642], we can choose  $\alpha \in \Gamma^{m}_n$ such that $a^t \notin G_\alpha$ for $1 \leq t <4n$.  Thus for such $\alpha$ we have $G_{\alpha}=\{1\}$ or $G_{\alpha}=\{1,a^{t}b,a^{-(2n-1)t}b\}$  for some $1 \leq t <4n$ since if  $G_{\alpha}\neq \{1\}$ and $a^{t_1} b,a^{t_2} b \in G_\alpha$, then
 $$
  a^{t_1} b.a^{t_2} b=a^{t_1}b.ba^{(2n-1)t_2} = a^{t_1+(2n-1)t_2}\in G_\alpha
 $$
 which shows that $t_1= -(2n-1)t_2$. To prove that  $\nu_2(\frac{h}{2n})< 0$ is a necessary condition for existence of orthogonal $\ast$-basis for $V_{\varsigma}(G)$, it is enough to consider both cases $G_{\alpha}= \{1\}$ and $G_{\alpha}=\{1,a^{t}b,a^{-(2n-1)t}b\}$. For both cases, we have
$$
\| e^{*}_\alpha\|^2= \frac{\varsigma(1)}{|G|}\sum_{g \in  G_\alpha}\varsigma(g)=\frac{2}{8n}=\frac{1}{4n}\neq 0,
$$
 so $\alpha \in \overline{\Delta}$.  First consider $G_{\alpha}= \{1\}$. For any $\sigma_1,\sigma_2 \in G$, we have
$$
\sigma_2 G_\alpha {\sigma_1~}^{-1}= \left\{\begin{array}{ll} \{ a^{i-j} \} ,\hspace{2.1cm}\text{if} \hspace{0.8cm}\sigma_1=a^j,\sigma_2=a^i \\
\{ a^{i+j(1-2n)}b \} ,\hspace{1.0cm}\text{if} \hspace{0.8cm}\sigma_1=a^jb,\sigma_2=a^i \\
\{  a^{(1-2n)i-j} \} ,\hspace{1.2cm}\text{if} \hspace{0.8cm}\sigma_1=a^jb,\sigma_2=a^ib
\end{array}\right.
$$
Therefore by (\ref{important2}) we have

$$
 \langle e^{*}_{\sigma_1.\alpha},e^{*}_{\sigma_2.\alpha}\rangle = \left\{\begin{array}{ll} \frac{1}{2n}\cos ( \frac{(i-j)h \pi}{2n}) ,\hspace{0.7cm}\text{if} \hspace{0.8cm}\sigma_1=a^j,\sigma_2=a^i \\
0 ,\hspace{3.0cm}\text{if} \hspace{0.8cm}\sigma_1=a^jb,\sigma_2=a^i \\
 \frac{1}{2n}\cos ( \frac{(i-j)h \pi}{2n}) ,\hspace{0.7cm}\text{if} \hspace{0.8cm} \sigma_1=a^jb,\sigma_2=a^ib
\end{array}\right.
$$

Hence  $\langle e^{*}_{\sigma_1.\alpha},e^{*}_{\sigma_2.\alpha}\rangle =0$ implies that there exist $ t_1$ and $t_2$ such that $\cos ( \frac{(t_1-t_2)h \pi}{2n})=0$, therefore by Lemma \ref{l5} we get $\nu_2(\frac{h}{2n})< 0$. Now consider $G_{\alpha}=\{1,a^{t}b,a^{-(2n-1)t}b\}$. For any $\sigma_1,\sigma_2 \in G$, we have
$$
\sigma_2 G_\alpha {\sigma_1~}^{-1}= \left\{\begin{array}{ll} \{ a^{i-j},   ba^{(2n-1)(j+t)-i}  ,   ba^{(2n-1)(j-(2n-1)t)-i}   \} ,\hspace{2.1cm}\text{if} \hspace{0.8cm}\sigma_1=a^j,\sigma_2=a^i \\
\{ a^{i+j(1-2n)}b, a^{j+(2n-1)t+i} , a^{j-t+i}   \} ,\hspace{3.7cm}\text{if} \hspace{0.8cm}\sigma_1=a^jb,\sigma_2=a^i \\
\{  a^{(1-2n)i-j},  a^{j+(2n-1)t+i}b,  a^{j-t+i}b                                          \} ,\hspace{3.6cm}\text{if} \hspace{0.8cm}\sigma_1=a^jb,\sigma_2=a^ib
\end{array}\right.
$$
Now similar to our previous calculations in this section, we get $\nu_2(\frac{h}{2n})< 0$.
\end{proof}
\begin{rem} In the proof of the necessary condition part of Theorem \ref{t1}, one can choose $\alpha=(1,2,2,..,2)$. The proof given here shows the stronger statement that the orbital subspace $V^{*}_\alpha$ looks an orthogonal $\ast$-basis whenever $G_{\alpha} \cup \langle a \rangle= \{1\}$.
\end{rem}
\begin{cor}\label{cor odd} Let $G =SD_{8n}$, $n$ is odd, be a subgroup of $S_{4n}$, denote $\varsigma=\varsigma_{h'}$ for $h' \in C^{\dag}_{even}$, and assume $d= \dim V\geq 2 $. Then $V_{\varsigma}(G)$ does not have an orthogonal $\ast$-basis.
\end{cor}
\begin{proof}
 Since $n$ is odd then $\nu_2(\frac{h'}{2n})\geq 0$. Thus using Theorem 3.4 $V_{\varsigma}(G)$ does not have an orthogonal $\ast$-basis.
\end{proof}

\begin{thm}\label{thm odd}
\label{t2}
Let $G =SD_{8n}$, be a subgroup of $S_{4n}$, denote $\psi=\psi_{h'}$ for $h' \in C^{\dag}_{odd}$ (even $n$) or $h' \in C^{odd}_{2,3}\setminus \{n,3n \}$(odd n) , and assume $d= \dim V\geq 2 $. Then $V_{\psi}(G)$ does not have orthogonal $\ast$-basis.
\end{thm}

\begin{proof} The proof is similar to the proof of Theorem \ref{t1}. Using Table I and Table
II we conclude that  $\langle e^{*}_{\sigma_1.\alpha},e^{*}_{\sigma_2.\alpha}\rangle \neq 0$ since the imaginary and real parts should both be equal to zero; but $i\sin x$ and $\cos x$ can not vanish simultaneously.
\end{proof}

\section{ACKNOWLEDGMENTS}
The authors are grateful to Professor Hjalmar Rosengren for valuable comments and for reviewing earlier drafts very carefully.


\bigskip
\bigskip

\address Mahdi Hormozi \\
{ Department of Mathematical Sciences, Division of Mathematics,\\ Chalmers University
of Technology and University of Gothenburg,\\ Gothenburg 41296, Sweden}\\
\email{hormozi@chalmers.se}\\

\address Kijti Rodtes\\
        {Department of  Mathematics\\
         Faculty of Science, Naresuan University\\
        Phitsanulok 65000, Thailand}\\
\email{kijtir@nu.ac.th}

\end{document}